\documentclass[11pt]{amsart}

\usepackage{amsmath, todonotes, enumitem, graphicx, tikz}

\usepackage[mathlines, modulo]{lineno}

\usepackage[colorlinks=false]{hyperref}

\usepackage{cleveref}

\theoremstyle{plain}
\newtheorem{theorem}{Theorem}[section]
\newtheorem{lemma}[theorem]{Lemma}
\newtheorem{proposition}[theorem]{Proposition}

\newtheorem{conjecture}[theorem]{Conjecture}
\newtheorem{problem}[theorem]{Problem}

\theoremstyle{definition}
\newtheorem{definition}[theorem]{Definition}
\newtheorem{claim}{Claim}[theorem]

\newcommand{\dash}{\nobreakdash-\hspace{0pt}}
\newcommand{\ba}{\backslash}
\newcommand{\graph}{
\tikz[baseline=-2.75pt]{
\node (A) at (210:4pt) [circle, fill=black, inner sep = 0.5pt] {};
\node (B) at (90:4pt) [circle, fill=black, inner sep = 0.5pt] {};
\node (C) at (330:4pt) [circle, fill=black, inner sep = 0.5pt] {};
\draw (A.center) -- (B.center) -- (C.center);
\draw (A.center) to[out=40, in=140] (C.center);
\draw (A.center) to[out=320, in=220] (C.center);
}
}

\newcommand{\mbb}[1]{\ensuremath{\mathbb{#1}}}
\newcommand{\mcal}[1]{\ensuremath{\mathcal{#1}}}
\newcommand{\id}{\ensuremath{\operatorname{id}}}
\newcommand{\trun}{\ensuremath{\operatorname{trun}}}

\title{Fractal classes of matroids}

\author[Mayhew]{Dillon Mayhew}
\address{School of Mathematics and Statistics,
Victoria University of Wellington,
New Zealand}
\email{dillon.mayhew@vuw.ac.nz}

\author[Newman]{Mike Newman}
\address{Department of Mathematics and Statistics,
University of Ottawa,
Ottawa,
Canada}
\email{mnewman@uottawa.ca}

\author[Whittle]{Geoff Whittle}
\address{School of Mathematics and Statistics,
Victoria University of Wellington,
New Zealand}
\email{geoff.whittle@vuw.ac.nz}

\date{\today}

\begin{document}

\begin{abstract}
A minor-closed class of matroids is (strongly) fractal if the number of
$n$\dash element matroids in the class is dominated by the
number of $n$\dash element excluded minors.
We conjecture that when \mbb{K} is an infinite field, the class of
\mbb{K}\dash representable matroids is strongly fractal.
We prove that the class of sparse paving matroids
with at most $k$ circuit-hyperplanes is a strongly fractal class
when $k$ is at least three.
The minor-closure of the class of spikes with at most $k$
circuit-hyperplanes (with $k\geq 5$) satisfies a strictly weaker
condition: the number of $2t$\dash element matroids in
the class is dominated by the number of
$2t$\dash element excluded minors.
However, there are only finitely many excluded minors
with ground sets of odd size.
\end{abstract}

\maketitle

\section{Introduction}

In \cite{MNW09} we proved that every real-representable matroid is
contained as a minor in an excluded minor for the class of real-representable
matroids.
(The same phenomenon holds for any infinite field.)
Geelen and Campbell strengthened this by showing that
every real-representable matroid is a minor of a complex-representable
excluded minor for the class of real-representable matroids \cite{CG18}.
In contrast to these results, the resolution of Rota's conjecture \cite{GGW14}
implies that there are only finitely many excluded minors for \mbb{F}\dash representability
when \mbb{F} is a finite field.

In this article we consider another possible dichotomy
between finite fields and infinite fields
in matroid representation theory.
Let \mcal{M} be a minor-closed class of matroids, and
let \mcal{EX} be the class of excluded minors for \mcal{M}.
For any non-negative integer $n$, let $m_{n}$ be the number of
non-isomorphic $n$\dash element matroids in \mcal{M}.
Thus $m_{n}$ counts the $n$\dash element members of \mcal{M}
modulo the equivalence relation of isomorphism.
Similarly, let $x_{n}$ be the number of non-isomorphic
$n$\dash element matroids in \mcal{EX}.
(Henceforth, when we refer to a matroid, we usually mean an isomorphism class
of matroids.)
We consider the probability that a matroid
chosen randomly from the $n$\dash element members of $\mcal{M}\cup \mcal{EX}$
is an excluded minor.
In other words, we consider the ratio
\[
\frac{x_{n}}{m_{n} + x_{n}}.
\]
We will denote this fraction by $\Gamma_{\mcal{M}}(n)$.
We consider only the case that \mcal{M} contains infinitely many matroids,
so that $\Gamma_{\mcal{M}}(n)$ is defined for all $n$.
If \mcal{M} has only finitely many excluded minors, then
$\Gamma_{\mcal{M}}(n)=0$ for all large enough values of $n$.
It is impossible for $\Gamma_{\mcal{M}}(n)$ to be equal to one,
but it could tend to one in the limit.
In this case our random choice is asymptotically certain to be
an excluded minor, so in some sense, the class \mcal{M} is
eventually overwhelmed by its `boundary': the set of excluded minors.
This leads us to the following terminology.

\begin{definition}
\label{father}
Let \mcal{M} be a minor-closed class of matroids.
If
\[\lim_{n\to\infty}\Gamma_{\mcal{M}}(n)= 1,\]
then \mcal{M} is a \emph{strongly fractal} class.
\end{definition}

If \mcal{M} is the class of \mbb{F}\dash representable matroids where
\mbb{F} is a finite field, then $\Gamma_{\mcal{M}}(n)=0$ for all large
enough values of $n$, since Rota's conjecture is true.
We believe that this fails for infinite fields in the strongest possible
way.

\begin{conjecture}
Let \mbb{K} be an infinite field.
The class of $\mathbb{K}$-representable matroids is strongly fractal.
\end{conjecture}

The class of gammoids is like the class of real-representable matroids,
in that the excluded minors form a maximal antichain \cite{May16}.
We conjecture that the class of gammoids is strongly fractal.

In this article we are content merely to establish the non-obvious
fact that strongly fractal classes exist.
A matroid is \emph{sparse paving} if every non-spanning circuit is a
hyperplane.

\begin{theorem}
\label{valley}
Let $k\geq 3$ be a positive integer.
Let $\mcal{P}_{k}$ be the class of sparse paving matroids with at most
$k$ circuit-hyperplanes.
Then $\mcal{P}_{k}$ is strongly fractal.
\end{theorem}

A rank\dash $r$ \emph{spike} has a ground set of size $2r$,
say $\{a_{1},b_{1},\ldots, a_{r},b_{r}\}$.
Assume that $r>3$.
Then the non-spanning circuits are exactly the sets of the form
$\{a_{i},b_{i},a_{j},b_{j}\}$, along with (possibly) some sets
that intersect each $\{a_{i},b_{i}\}$ in either $a_{i}$ or $b_{i}$.
Any circuit of the latter type is also a hyperplane.
Sometimes such a matroid is called a \emph{tipless} spike.
The class of spikes with a bounded number of circuit-hyperplanes
is not minor-closed, but if we close it under minors, we obtain
a class that has a weaker fractal property.

\begin{definition}
\label{master}
Let \mcal{M} be a minor-closed class of matroids.
If $\Gamma_{\mcal{M}}(1), \Gamma_{\mcal{M}}(2), \Gamma_{\mcal{M}}(3),\ldots$
contains an infinite subsequence that converges to one, then \mcal{M} is
\emph{weakly fractal}.
\end{definition}

Obviously a strongly fractal class is weakly fractal.

\begin{theorem}
\label{memory}
Let $k \geq 5$ be an integer.
Let $\mcal{S}_{k}$ be the class produced by taking minors of the
spikes with at most $k$ circuit-hyperplanes.
Then $\mcal{S}_{k}$ is a weakly fractal class, but is not strongly fractal.
\end{theorem}

The reason the class in \Cref{memory} is not strongly fractal is that
eventually there are no excluded-minors with odd-cardinality ground sets
(\Cref{cheese}).
In other words, $\Gamma_{\mcal{S}_{k}}(2t+1)=0$ for all large-enough
values of $t$.
However $\Gamma_{\mcal{S}_{k}}(2t)$ converges to one.

\Cref{master} does not require the
subsequence that converges to one to have any particular
structure.
However, it is inconceivable that the following conjecture fails.

\begin{conjecture}
\label{drawer}
Let \mcal{M} be a weakly fractal class of matroids.
There exist integers $a$ and $b$ such that the
sequence
\[
\Gamma_{\mcal{M}}(a+b),\ 
\Gamma_{\mcal{M}}(2a+b),\ 
\Gamma_{\mcal{M}}(3a+b),\ldots
\]
converges to one.
\end{conjecture}

In the case of the class $\mcal{S}_{k}$, the integers
$a=2$ and $b=0$ satisfy \Cref{drawer}.

The only fractal classes we know of contain infinite
antichains, and we think this exemplifies a general pattern.

\begin{conjecture}
Any weakly fractal class of matroids contains an infinite antichain.
\end{conjecture}

We use $\mcal{P}(X)$ to denote the power set of the set $X$.
The symbol $\mbb{Z}_{\geq 0}$ stands for the set of non-negative
integers.
Our reference for matroid terms and concepts is \cite{Oxl11}.
Recall that a \emph{triangle} is a circuit of size three, and a
\emph{triad} is a $3$\dash element cocircuit.
A \emph{parallel pair} is a circuit of size two,
and a \emph{parallel class} is a maximal set such that every
$2$\dash element subset is a parallel pair.
A \emph{series pair} is a $2$\dash element cocircuit, and
a \emph{series class} is a maximal set with every $2$\dash element
subset being a series pair.
A \emph{thin} edge of a graph is a non-loop edge that is not in any
parallel pair.
Let $f(n)$ and $g(n)$ be functions taking integers as input and returning real
numbers as output.
If we say that $f(n)$ is bounded by $O(g(n))$, we mean that there
exist a constant $c$ and an integer $N$ such that $f(n)\leq cg(n)$ whenever
$n>N$.
Similarly, if $f$ is at least $\Omega(g(n))$, then there exist
$c$ and $N$ such that $f(n)\geq cg(n)$ whenever $n>N$.

\section{Sparse paving matroids}

A matroid is \emph{sparse paving} if every non-spanning
circuit is also a hyperplane.
Every matroid with rank or corank equal to zero is vacuously sparse paving.
Let $E$ be an $n$\dash element set, and let $r$ be an integer satisfying
$1\leq r\leq n-1$.
Rank\dash $r$ sparse paving matroids on the ground set $E$
are in bijective correspondence with families, \mcal{C}, of
$r$\dash element subsets of $E$, such that if
$C_{1}$ and $C_{2}$ are distinct members of \mcal{C}, then
$|C_{1} - C_{2}| > 1$.
If \mcal{C} is such a collection, then we use $M(\mcal{C})$ to denote the
sparse paving matroid on the ground set $E$ with \mcal{C} as its family of
circuit-hyperplanes.

For $k$ a non-negative integer, let $\mcal{P}_{k}$ denote the class of
sparse paving matroids with at most $k$ circuit-hyperplanes.
Note that $\mcal{P}_{0}$ is the class of uniform matroids.
Let \mcal{P} be the set of all sparse paving matroids.
Thus $\mcal{P} = \cup_{k\geq 0}\mcal{P}_{k}$.
Let $M=M(\mcal{C})$ be a sparse paving matroid on the ground set $E$,
and let $e$ be an element of $E$.
If $e$ is not a coloop in $M$ then
\[
M\ba e = M(\{C\colon C \in \mcal{C},\ e\notin C\})
\]
and if $e$ is not a loop, then
\[
M/e = M(\{C-e\colon C\in \mcal{C},\ e\in C\}).
\]
This demonstrates that $\mcal{P}_{k}$ is
a minor-closed class for any $k\geq 0$,
and so is \mcal{P}.

\begin{definition}
\label{crutch}
Let $k$ and $n$ be positive integers.
Let $\mcal{C}=(C_{1},\ldots, C_{k})$ be a sequence of subsets of $E$,
where $E$ is a set of cardinality $n$.
For every subset $I\subseteq \{1,\ldots, k\}$, define $\mcal{C}(I)$ to be
\[
\left(\bigcap_{i\in I} C_{i}\right)
\cap \left(\bigcap_{i\in \{1,\ldots, k\}-I}(E-C_{i})\right).
\]
Thus $\mcal{C}(I)$ contains those elements of $E$ that are in
$C_{i}$ for every $i\in I$, and in the complement of $C_{i}$
for every $i\notin I$.
This means that $(\mcal{C}(I))_{I\subseteq \{1,\ldots, k\}}$ is a
partition of $E$ (possibly containing empty blocks).
We think of $\mcal{C}(I)$ as being a cell in a Venn diagram with
$k$ sets.

Let $\psi_{\mcal{C}}$ be the function that takes $I$ to $|\mcal{C}(I)|$
for each $I\subseteq \{1,\ldots, k\}$.
Note that
\[
\sum_{I\subseteq \{1,\ldots, k\}}\psi_{\mcal{C}}(I) = n.
\]
\end{definition}

Let $k$ and $n$ be positive integers.
We describe a multi-valued function, $\mcal{R}_{k}^{n}$.
The domain of $\mcal{R}_{k}^{n}$ is the set of
$n$\dash element sparse paving matroids
with exactly $k$ circuit-hyperplanes.
The codomain is the set of functions from
$\mcal{P}(\{1,\ldots, k\})$ to $\mbb{Z}_{\geq 0}$.
The ordered pair $(M,\psi)$ is in $\mcal{R}_{k}^{n}$ if there is
some ordering $\mcal{C}=(C_{1},\ldots, C_{k})$ of the
circuit-hyperplanes of $M$ such that $\psi=\psi_{\mcal{C}}$.
We use $\mcal{R}_{k}^{n}(M)$ to denote
$\{\psi\colon (M,\psi)\in \mcal{R}_{k}^{n}\}$, the image of $M$ under
$\mcal{R}_{k}^{n}$.
Thus $\mcal{R}_{k}^{n}(M)$ contains
at most $k!$ functions.

\begin{proposition}
\label{sonnet}
Let $k$ and $n$ be positive integers.
Let $M$ and $N$ be $n$\dash element sparse paving matroids
with exactly $k$ circuit-hyperplanes.
Then $M$ and $N$ are isomorphic if and only if
$\mcal{R}_{k}^{n}(M)\cap \mcal{R}_{k}^{n}(N)\ne \emptyset$.
\end{proposition}

\begin{proof}
Let $\rho$ be an isomorphism from $M$ to $N$.
Let $\mcal{C}=(C_{1},\ldots, C_{k})$ be an ordering of the
circuit-hyperplanes in $M$.
Then $\rho(\mcal{C}) = (\rho(C_{1}),\ldots, \rho(C_{k}))$
is an ordering of the circuit-hyperplanes in $N$.
Because $\rho$ is a bijection, it is clear that
$|\mcal{C}(I)| = |\rho(\mcal{C})(I)|$ for every
$I\subseteq \{1,\ldots, k\}$.
Thus $\psi_{\mcal{C}} = \psi_{\rho(\mcal{C})}$, and it follows that
$\mcal{R}_{k}^{n}(M)\cap \mcal{R}_{k}^{n}(N)$ is not empty.

For the converse, we assume that
$\mcal{R}_{k}^{n}(M)\cap \mcal{R}_{k}^{n}(N)$ contains at least one
function.
This means that there must be orderings, $\mcal{C}_{M}$ and
$\mcal{C}_{N}$, of the circuit-hyperplanes in $M$ and $N$,
respectively, such that $\psi_{\mcal{C}_{M}} = \psi_{\mcal{C}_{N}}$.
For each $I\subseteq \{1,\ldots, k\}$, let $\pi_{I}$ be an arbitrary
bijection from $\mcal{C}_{M}(I)$ to $\mcal{C}_{M}(I)$.
Note that these sets have the same cardinality since
$\psi_{\mcal{C}_{M}}(I) = \psi_{\mcal{C}_{N}}(I)$, so $\pi_{I}$ exists.
We consider each $\pi_{I}$ to be a set of ordered pairs.
Now we can define $\pi$ to be
\[\bigcup_{I\subseteq \{1,\ldots, k\}}\pi_{I}.\]
Thus $\pi$ is a bijection from $E(M)$ to $E(N)$.
It is clear that $\pi(X)$ is a circuit-hyperplane of $N$
if and only if $X$ is a circuit-hyperplane of $M$.
Therefore $\pi$ is the desired isomorphism between $M$ and $N$.
\end{proof}

\begin{lemma}
\label{asylum}
Let $k$ be a non-negative integer.
The number of $n$\dash element matroids in $\mcal{P}_{k}$ is at most
$O(n^{2^{k}-1})$.
\end{lemma}

\begin{proof}
We observe that
\[
\mcal{P}_{k} = (\mcal{P}_{k} - \mcal{P}_{k-1})\cup
(\mcal{P}_{k-1} - \mcal{P}_{k-2})\cup\cdots\cup
(\mcal{P}_{1} - \mcal{P}_{0})\cup\mcal{P}_{0}.
\]
Because $\mcal{P}_{0}$ is the class of uniform matroids, it follows that
the number of $n$\dash element matroids in $\mcal{P}_{0}$ is at most
$O(n)$.
We will show that the number of $n$\dash element matroids in
$\mcal{P}_{m}-\mcal{P}_{m-1}$ is at most $O(n^{2^{m}-1})$, and then the result will follow,
since the number of classes in the above union is constant relative to $n$.

From \Cref{sonnet}, it follows that for any positive $m$, the number of $n$\dash element
matroids in $\mcal{P}_{m}-\mcal{P}_{m-1}$ is at most the number of functions
$\psi\colon \mcal{P}(\{1,\ldots, m\})\to \mbb{Z}_{\geq 0}$ such that
$\sum_{I\subseteq \{1,\ldots, m\}}\psi(I) = n$.
By standard enumeration techniques, the number of such functions is
\[
\binom{n+2^{m}-1}{n} = \binom{n+2^{m}-1}{2^{m}-1}.
\]
Since $m$ is a constant, this binomial coefficient is bounded by $O(n^{2^{m}-1})$,
and we are done.
\end{proof}

\begin{lemma}
\label{collar}
Let $k\geq 3$ be an integer.
The number of $n$\dash element
excluded minors for $\mcal{P}_{k}$ is
at least $\Omega(n^{2^{k+1}-k-4})$.
\end{lemma}

\begin{proof}
Let \mcal{I} be the collection $\{I\subseteq \{1,\ldots, k+1\}\colon 2\leq I\leq k\}$.
Observe that $|\mcal{I}| = 2^{k+1}-k-3$.
For $s\in\{2,\ldots, k\}$, let $\mcal{I}^{s}$ denote the collection
of sets in \mcal{I} with cardinality $s$.
For each $I\in\mcal{I}$ we introduce a variable $x_{I}$.
We are going to consider non-negative integer solutions to the
equation
\begin{equation}
\label{eqn1}
k\sum_{I\in\mcal{I}^{2}}x_{I}
+(k-1)\sum_{I\in\mcal{I}^{3}}x_{I}
+\cdots+
2\sum_{I\in\mcal{I}^{k}}x_{I}
=n-2(k+1)
\end{equation}

\begin{claim}
\label{trench}
The number of non-negative integer solutions to \eqref{eqn1}
is at least $\Omega(n^{2^{k+1}-k-4})$.
\end{claim}

\begin{proof}
The proof of this \namecref{trench} is essentially the same as the proof of
Schur's Theorem given in \cite[Theorem 3.15.2]{Wil94}.
By standard techniques, we see that the number of
non-negative integer solutions is equal to the coefficient of
$z^{n-2(k+1)}$ in the generating function
\[
f(z)=
\left(\frac{1}{1-z^{k}}\right)^{\binom{k+1}{2}}
\left(\frac{1}{1-z^{k-1}}\right)^{\binom{k+1}{2}}\cdots
\left(\frac{1}{1-z^{2}}\right)^{\binom{k+1}{k}}.
\]
Every pole of $f(z)$ is a root of unity.
In particular, the denominator of $f(z)$ has as a factor
\[
(1-z)^{\binom{k+1}{2}+\binom{k+1}{3}+\cdots + \binom{k+1}{k}} = (1-z)^{2^{k+1}-k-3}
\]
which shows that $z=1$ is a pole with multiplicity
$2^{k+1}-k-3$.
If $s$ is an integer greater than one, then $s$ does not divide all the
values in $2,\ldots, k$.
In this case, if $\omega$ is an $s$\dash th root of unity and $z=\omega$ is a
pole of $f(z)$, then its multiplicity is less than $2^{k+1}-k-3$.
So $f(z)$ has a pole of multiplicity $2^{k+1}-k-3$ at $z=1$, and the
multiplicity of every other pole is less than this value.
Now the arguments in \cite[Theorem 3.15.2]{Wil94} shows that
the number of non-negative integer solutions is asymptotically
equal to $n^{2^{k+1}-k-4}$, and this gives us the desired result.
\end{proof}

Let $\phi$ be an arbitrary solution to \eqref{eqn1}.
Thus $\phi$ takes the variables $\{x_{I}\}_{I\in\mcal{I}}$ to non-negative
integers, and 
\[
k\sum_{I\in\mcal{I}^{2}}\phi(x_{I})
+(k-1)\sum_{I\in\mcal{I}^{3}}\phi(x_{I})
+\cdots+
2\sum_{I\in\mcal{I}^{k}}\phi(x_{I})
=n-2(k+1).
\]
We will construct a sequence $\mcal{C}=(C_{1},\ldots, C_{k+1})$,
of subsets of $\{1,\ldots, n\}$ such that:
\begin{enumerate}[label=\textup{(\roman*)}]
\item $C_{1},\ldots, C_{k+1}$ are equicardinal,
\item  $|C_{i}-C_{j}|>1$ when $i$ and $j$ are distinct,
\item $\psi_{\mcal{C}}(I) = \phi(x_{I})$ for every $I\in \mcal{I}$, and
\item the sparse paving matroid $M(\mcal{C})$ is an $n$\dash element excluded minor for
$\mcal{P}_{k}$.
\end{enumerate}

We construct $\mcal{C}=(C_{1},\ldots, C_{k+1})$ by allocating each element
in $\{1,\ldots, n\}$ to a unique set of the form $\mcal{C}(I)$ for
some $I\subseteq \{1,\ldots, k+1\}$.
We start by allocating two elements to each set of the form
$\mcal{C}(\{i\})$, for $i\in \{1,\ldots, k+1\}$.
This ensures that statement (ii) holds.
We now have $n-2(k+1)$ elements left to allocate.
We will allocate no elements to $\mcal{C}(\emptyset)$ or $\mcal{C}(\{1,\ldots, k+1\})$,
so every element is in at least one of the sets $(C_{1},\ldots, C_{k+1})$,
and no element is in all of them.

We process each subset $I\in \mcal{I}$ in turn.
We allocate $\phi(x_{I})$ elements to
$\mcal{C}(I)$, and then for each $i\in \{1,\ldots, k+1\} - I$, we allocate a
further $\phi(x_{I})$ elements to $C(\{i\})$.
We have thus allocated an additional $\phi(x_{I})$ elements to each set in
$(C_{1},\ldots, C_{k+1})$, ensuring the sets remain equicardinal during this process.
Note that the number of elements we have allocated while processing
$I$ is $\phi(x_{I})+((k+1)-|I|)\phi(x_{I})$.
After processing every subset in \mcal{I}, the number of elements we have allocated is
therefore
\[
k\sum_{I\in\mcal{I}^{2}}\phi(x_{I})
+(k-1)\sum_{I\in\mcal{I}^{3}}\phi(x_{I})
+\cdots+
2\sum_{I\in\mcal{I}^{k}}\phi(x_{I}) = n-2(k+1).
\]
Hence all $n$ elements have now been allocated, and the sets
$(C_{1},\ldots, C_{k+1})$ are equicardinal, and satisfy $|C_{i}-C_{j}|>1$
when $i$ and $j$ are distinct.
Furthermore, our method of construction ensures that
$\psi_{\mcal{C}}(I) = \phi(x_{I})$ for every $I\in \mcal{I}$.
Since each element $e\in \{1,\ldots, n\}$ is in at least one of the sets
in \mcal{C}, but not all of them, it follows that $M(\mcal{C})\ba e$ and
$M(\mcal{C})/e$ both have at most $k$ circuit-hyperplanes, while
$M(\mcal{C})$ itself has $k+1$.
Thus $M(\mcal{C})$ is an excluded minor for $\mcal{P}_{k}$,
as desired.

The number of excluded minors we have constructed in this way is
$\Omega(n^{2^{k+1}-k-4})$ by \Cref{trench}.
Some of these excluded minors may be isomorphic copies of the same matroid.
But because $\mcal{R}_{k+1}^{n}(M)$ is no larger than $(k+1)!$ for
any excluded minor $M$, \Cref{sonnet} implies that
any isomorphism class of excluded minors corresponds to no more
than $(k+1)!$ solutions to \eqref{eqn1}.
As $k$ is fixed with respect to $n$, dividing a function that is at least
$\Omega(n^{2^{k+1}-k-4})$ by $(k+1)!$ produces another such function,
so the proof of \Cref{collar} is complete.
\end{proof}

\begin{proof}[Proof of \textup{\Cref{valley}}.]
From \Cref{asylum,collar}, it follows that there are constants $c_{1}$ and
$c_{2}$ such that for sufficiently large values of $n$ we have
\[
\Gamma_{\mcal{P}_{k}}(n)\geq
\frac{c_{1}n^{2^{k+1}-k-4}}{c_{2}n^{2^{k}-1}+c_{1}n^{2^{k+1}-k-4}}
=\frac{1}{(c_{2}/c_{1})n^{-2^{k}+k+3}+1}.
\]
Since $k\geq 3$, it follows that
$-2^{k}+k+3$ is negative, and hence
$\Gamma_{\mcal{P}_{k}}(n)$ tends to one as $n$ tends to infinity.
\end{proof}

\section{Spikes}

We describe spikes and their minors using \emph{biased graphs}.
Let $G$ be an undirected graph, which may contain loops and parallel edges.
A \emph{theta-subgraph} of $G$ consists of two distinct vertices, $u$ and $v$,
and three paths from $u$ to $v$ that do not share any vertices other than $u$ and $v$.
A \emph{linear class} is a collection, \mcal{B}, of cycles of $G$ satisfying the constraint
that no theta-subgraph of $G$ contains exactly two cycles in \mcal{B}.
In this case, we say that the pair $(G,\mcal{B})$ is a \emph{biased graph}.
The cycles in \mcal{B} are \emph{balanced} and any other cycle is \emph{unbalanced}.
A subgraph is \emph{unbalanced} if it contains an unbalanced
cycle, and otherwise it is \emph{balanced}.
Let $E$ be the edge-set of $G$.
We similarly say that $X\subseteq E$ is balanced if the
subgraph $G[X]$ is balanced, and otherwise $X$ is unbalanced.

\emph{Lift matroids} were introduced by Zaslavsky \cite{Zas91}.
The lift matroid, $L(G,\mcal{B})$, has $E$ as its ground set.
The set $X\subseteq E$ is a circuit of $L(G,\mcal{B})$ if and only if $G[X]$
is either: (i) a balanced cycle, (ii) an unbalanced theta-subgraph, or
(iii) a pair of unbalanced cycles with at most one vertex in common.
The rank of $L(G,\mcal{B})$ is equal to the number of vertices in $G$,
minus the number of balanced connected components.
A set, $X\subseteq E$, is a hyperplane of $L(G,\mcal{B})$ if and only if
$X$ is a maximal balanced set, or is unbalanced and is a hyperplane
of the graphic matroid $M(G)$ \cite[Theorem 3.1]{Zas91}.
Naturally, this characterises the cocircuits of $L(G,\mcal{B})$.

The next result is well known, but we include the
proof for completeness.

\begin{proposition}
\label{ritual}
Let $e$ be an element of the matroid $M$, and assume that $M/e = M(G)$
for some graph $G$.
Let $G^{e}$ be the graph obtained from $G$ by adding the loop $e$ incident
with an arbitrary vertex.
Let \mcal{B} be the collection of  cycles in $G$ such that $C\in \mcal{B}$
if and only if the edge-set of $C$ is a circuit of $M$.
Then \mcal{B} is a linear class of $G^{e}$ and $M=L(G^{e},\mcal{B})$.
\end{proposition}

\begin{proof}
Let $X$ be a set of edges such that $G[X]$ is a theta-subgraph.
Assume $G[X]$ contains two cycles in \mcal{B}.
Hence there are distinct circuits $C_{1}$ and $C_{2}$ of $M$ that are contained
in $X$.
Two cycles in a theta-subgraph must contain a common edge,
so we assume that $x$ is in $C_{1}\cap C_{2}$, and that
therefore $(C_{1}\cup C_{2})-x$ contains a circuit, $C_{3}$, of $M$.
Note that $C_{3}\subseteq X$, and $C_{3}$ is a union of circuits in
$M/e = M(G)$.
But $G[X]$ contains exactly one cycle that does not contain $x$:
the third cycle in the theta-subgraph.
Therefore $C_{3}$ is the edge-set of this cycle, which implies that $G[X]$
contains three cycles in \mcal{B}.
This shows that \mcal{B} is a linear class.

Let $C$ be a circuit of $L(G^{e},\mcal{B})$.
If $G[C]$ is a balanced cycle, then $C$ is also a circuit of $M$.
Assume that $G[C]$ consists of two unbalanced cycles that share at
most one vertex.
If one of these cycles is the loop $e$, then $C$ is also a circuit in $M$,
so we will assume $e$ is not in $C$.
Then there are disjoint circuits $C_{1}$ and $C_{2}$ in $M/e$
such that $C=C_{1}\cup C_{2}$ and both $C_{1}\cup e$ and
$C_{2}\cup e$ are circuits of $M$.
By circuit elimination, $C_{1}\cup C_{2}$ contains a circuit,
$C_{3}$, of $M$.
Hence $C_{3}$ is a union of circuits in $M/e$.
But any circuit of $M/e=M(G)$ contained in
$C_{3}$ is either $C_{1}$ or $C_{2}$, so $C_{3}$ is either
equal to one of these circuits, or to their union.
The first case is impossible, as $C_{1}$ and $C_{2}$ are not circuits of
$M$.
Hence $C_{3}=C_{1}\cup C_{2} = C$ is also a circuit in $M$.

Now assume that $C$ is the edge-set of an unbalanced theta-subgraph.
Let $C_{1}$ and $C_{2}$ be the edge-sets of two distinct cycles
in the theta-subgraph.
As before, $C_{1}\cup C_{2}$ contains a circuit, $C_{3}$, of $M$,
and $C_{3}$ is a union of circuits in $M/e=M(G)$.
Since $C_{1}$ and $C_{2}$ are not circuits of $M$,
there are only two possibilities: $C_{3}$ is the edge-set of the third cycle in the
theta-subgraph (which is impossible as the theta-subgraph is unbalanced),
or $C_{3}$ is the entire theta-subgraph.
We deduce that $C$ is also a circuit in $M$.

Now we know that every circuit of $L(G^{e},\mcal{B})$ is also a
circuit in $M$, so to complete the proof it suffices to show that every circuit of
$M$ contains a circuit of $L(G^{e},\mcal{B})$.
Let $C$ be a circuit of $M$.
If $e$ is in $C$, then $C-e$ is a circuit of $M/e=M(G)$, so
$C-e$ is the edge-set of an unbalanced cycle.
In this case $C$ is the union of two unbalanced cycles,
one of them being the loop $e$, so $C$ is a circuit in $L(G^{e},\mcal{B})$.
Hence we assume $e\notin C$.
Now $C$ is the edge-set of a union of cycles in $G$.
If any of these is a balanced cycle, then $C$ contains a circuit
in $L(G^{e},\mcal{B})$.
If the union contains only one unbalanced cycle, then $C$ is independent
in $M$, a contradiction.
Thus the union contains at least two unbalanced cycles.
It is now easy to see that it therefore contains a theta-subgraph,
or two cycles that share at most one vertex.
Thus $C$ contains a circuit of $L(G^{e},\mcal{B})$ and we are done.
\end{proof}

\begin{definition}
\label{bounce}
Let $r\geq 3$ be an integer, and let $\Delta_{r}$ be the graph obtained
from a cycle with $r$ edges by replacing each edge with a parallel pair.
A \emph{(tipless) spike} is a matroid of the form $L(\Delta_{r},\mcal{B})$, where
\mcal{B} is a linear class of Hamiltonian cycles.
Let \mcal{S} denote the class of matroids that are isomorphic to minors of spikes.
Let $k$ be a non-negative integer.
We use $\mcal{S}_{k}$ to denote the class of matroids that are isomorphic
to minors of spikes of the form $L(\Delta_{r},\mcal{B})$,
where \mcal{B} contains at most $k$ Hamiltonian cycles.
Therefore $\mcal{S} = \cup_{k\geq 0}\ \mcal{S}_{k}$.
\end{definition}

Recall that a \emph{cyclic flat} is a flat that is a (possibly empty) union of circuits.
A set $X$ is dependent if and only if $|X\cap Z| > r(Z)$ for some
cyclic flat $Z$, so any matroid is determined by its cyclic flats and their ranks.
It is an easy exercise to prove the following result.

\begin{proposition}
\label{charge}
Let $r\geq 3$ be an integer, and let \mcal{B} be a linear class
of Hamiltonian cycles in $\Delta_{r}$.
The cyclic flats of $L(\Delta_{r},\mcal{B})$ are as follows:
\begin{enumerate}[label=\textup{(\roman*)}]
\item the ground set is a cyclic flat of rank $r$,
\item the empty set is a cyclic flat of rank zero,
\item the edge-set of each cycle in \mcal{B} is a cyclic flat of rank $r-1$,
\item Any set of $p$ parallel pairs is a cyclic flat of rank $p+1$,
when $2\leq p \leq r-2$.
\end{enumerate}
\end{proposition}

Let $r\geq 3$ be an integer, and let $C$ be a Hamiltonian cycle of $\Delta_{r}$.
Let $C^{*}$ be the Hamiltonian cycle that contains no edges in common with $C$.
If \mcal{B} is a linear class of Hamiltonian cycles, then $\mcal{B}^{*}$ is
the linear class $\{C^{*}\colon C\in \mcal{B}\}$.
It is well-known that the cyclic flats of $M^{*}$ are exactly the complements
of cyclic flats of $M$.
Now the next result is easy to check.

\begin{proposition}
\label{attack}
Let $r\geq 3$ be an integer, and let \mcal{B} be a linear class of
Hamiltonian cycles of $\Delta_{r}$.
Then $(L(\Delta_{r},\mcal{B}))^{*} = L(\Delta_{r},\mcal{B}^{*})$.
Consequently, $\mcal{S}_{k}$  is closed under duality
for each $k\geq 0$, and so is \mcal{S}.
\end{proposition}

Although the set of spikes is not closed under taking minors,
we are able to give an explicit description of all the matroids
in $\mcal{S}_{k}$.
This is accomplished in \Cref{safety}, which we now move towards proving.

The following description of minor operations on lift matroids follows from
\cite[Theorem 3.6]{Zas91}.
Let \mcal{B} be a linear class of cycles in the graph $G$, and
let $e$ be an edge of $G$.
We define $\mcal{B}\ba e$ to be the collection of cycles in \mcal{B} that do not
contain $e$.
Then $L(G,\mcal{B})\ba e = L(G\ba e,\mcal{B}\ba e)$.
If $e$ is not a loop, then we define $\mcal{B}/e$ to be the collection of
cycles in $G/e$ with edges sets of the form $E(C)-e$, where $C$ is a cycle in \mcal{B}
that may or may not contain $e$.
With this definition, the equality
$L(G,\mcal{B})/e = L(G/e,\mcal{B}/e)$ holds.
If $e$ is a balanced loop, then $L(G,\mcal{B})/e = L(G,\mcal{B})\ba e$,
and if $e$ is an unbalanced loop, then
$L(G,\mcal{B})/e$ is equal to $M(G\ba e)$, the cycle matroid of $G\ba e$.
Since any cycle matroid can be expressed as a lift matroid
(by making every cycle balanced), these observations show that the class of lift matroids
is minor-closed.

We recall that graphs may contain loops and parallel edges.
Let \mcal{G} be the class of graphs containing:
\begin{enumerate}[label=\textup{(\roman*)}]
\item any graph with a single vertex,
\item any connected graph with exactly two vertices, and at most four edges
joining them,
\item any graph whose underlying simple graph is a cycle of at least three
vertices, where each parallel class contains at most two edges.
\end{enumerate}

We note that if two graphs in \mcal{G} have the same parallel pairs,
loops, and thin edges, then their Hamiltonian cycles have the
same edge-sets.
Furthermore the lift matroids corresponding to identical
linear classes are equal.
In other words, the cyclic order in which the parallel pairs and
thin edges appear is immaterial to the lift matroid.

\begin{proposition}
\label{safety}
Let $k$ be a non-negative integer.
A matroid belongs to $\mcal{S}_{k}$ if and
only if it satisfies at least one of the following statements.
\begin{enumerate}[label=\textup{(\Alph*)}]
\item $M = L(G,\mcal{B})$, where $G\in \mcal{G}$ has at least
three vertices, and \mcal{B} is a linear class of at most $k$ Hamiltonian cycles,
\item $M = L(G,\mcal{B})$, where $G\in \mcal{G}$ has exactly two vertices, 
and \mcal{B} is a linear class of at most $k$ edge-disjoint Hamiltonian cycles,
\item $M = L(G,\mcal{B})$, where $G\in \mcal{G}$ has a single
vertex, and \mcal{B} contains at most $\min\{k,1\}$ loops,
\item $M=M(G)$ for a graph $G\in \mcal{G}$,
\item $M=M^{*}(G)$ for a graph $G\in \mcal{G}$, or
\item every connected component of $M$ has size at most two.
\end{enumerate}
\end{proposition}

\begin{definition}
We refer to matroids satisfying the statements in \Cref{safety} as being
\emph{Category\dash}(A), (B), (C), (D), (E), or (F), respectively.
\end{definition}

\begin{proof}[Proof of \textup{\Cref{safety}}.]
We start by proving that any matroid in $\mcal{S}_{k}$ satisfies one of the
statements in \Cref{safety}.
Assume this fails for $M\in \mcal{S}_{k}$.
Now $M$ can be expressed as $L(\Delta_{r},\mcal{B})/I\ba J$ for
disjoint sets $I$ and $J$, where \mcal{B} contains at most $k$
Hamiltonian cycles.
We assume that we have chosen $M$ so that $|I\cup J|$ is as small as
possible.
If $|I \cup J|=0$, then $M=L(\Delta_{r},\mcal{B})$ is a Category\dash (A)
matroid, which is impossible.
Therefore we let $e$ be an element in $I\cup J$, and we define
$M^{e}$ to be $L(\Delta_{r},\mcal{B})/(I-e)\ba (J-e)$.
Note that $M^{e}$ is in $\mcal{S}_{k}$, and $M$ is either
$M^{e}/e$ or $M^{e}\ba e$.
Our choice of $M$ means that $M^{e}$ is not a counterexample
to the \namecref{safety}.

If $M^{e}$ is Category\dash (F), then so is $M$,
which is impossible.
Assume that $M^{e}=M(G)$ is Category\dash (D).
Then $M$ is also Category\dash (D) unless
$M = M^{e}\ba e$ where $e$ is a thin edge in $G$.
But in this case any circuit of $M$ is either a loop,
or a parallel pair in $G\ba e$.
Hence $M$ is Category\dash (F).
The case when $M^{e}$ is Category\dash (E) leads to
a dual contradiction.
It is easy to see that any minor of a Category\dash (C) matroid
belongs to Category\dash (C) or (F).
Assume that $M^{e}=L(G,\mcal{B})$ is Category\dash (B).
If $e$ is a thin edge in $G$ and $M=M^{e}\ba e$,
then $M$ is a rank\dash one matroid with no loops, and is
therefore Category\dash (C).
In any other case $M^{e}\ba e$ is Category\dash (B),
so we assume that $M=M^{e}/e$.
If $e$ is a loop then it is unbalanced, and $M = M(G\ba e)$.
In this case $M$ is Category\dash (D).
So $e$ is a non-loop edge.
Since $e$ is in at most one balanced cycle,
$\mcal{B}/e$ contains at most one balanced loop.
Therefore $M = L(G/e,\mcal{B}/e)$ is Category\dash (C).

Now we must assume that $M^{e}=L(G,\mcal{B})$ is Category\dash (A).
Assume $M=M^{e}\ba e$.
If $e$ is not a thin edge, then $M$ is also Category\dash (A).
In the case that $e$ is a thin edge, there are no cycles in $\mcal{B}\ba e$.
The thin edges of $G\ba e$
are coloops in $M = L(G\ba e,\mcal{B}\ba e)=L(G\ba e,\emptyset)$.
The only circuits of $M$ consist of a pair of loops in $G$, a
loop and a parallel pair, or a pair of parallel pairs.
Now it is easy to see that $M$ is $M^{*}(H)$, where $H$ is in \mcal{G},
and has the same parallel pairs as $G\ba e$.
The loops of $H$ are the thin edges of $G\ba e$, and the thin edges of
$H$ are the loops of $G\ba e$.
Therefore $M$ is Category\dash (E).
Thus we assume that $M=M^{e}/e$.
If $e$ is a loop of $G$, then $M = M(G\ba e)$ is Category\dash (D),
so we assume $e$ is not a loop.
If $G$ has more than three vertices, then $M$ is certainly
Category\dash (A), so we assume $G$ has exactly three vertices.
To show that $M$ is Category\dash (B), we assume for a
contradiction that two cycles in $\mcal{B}/e$ share an edge.
This means that two cycles in \mcal{B} have two
common edges, one of which is $e$.
Given that $G$ has three vertices, these two cycles differ
in only one edge, which is a contradiction as \mcal{B} is a linear class.
We have now shown that matroids in $\mcal{S}_{k}$ satisfy at
least one of the statements in the \namecref{safety}.

Now we prove the converse.
Let $M$ be a Category\dash (F) matroid with $l$ loops,
 $p$ parallel pairs, and $c$ coloops.
Then $M$ is isomorphic to $M(G)\ba e$, where $G$ is a graph in \mcal{G}
with $l$ loops, $p$ parallel pairs, and $c+1$ thin edges, one of which is $e$.
This shows that every Category\dash (F) matroid is a minor of a
Category\dash (D) matroid.
Let $M = M(G)$ be a Category\dash (D) matroid.
Then $M = L(G^{e},\emptyset)/e$, where $G^{e}$ is obtained from $G$
by adding $e$ as a loop.
Since $G$ is in \mcal{G}, it follows that $G^{e}$ is also.
Thus every Category\dash (D) matroid is a minor of a
matroid in Category\dash (A), (B), or (C).
Similarly, let $M=M^{*}(G)$ be a Category\dash (E) matroid,
where $G$ has $l$ loops, $p$ parallel pairs, and $c$
thin edges.
Then $M$ is isomorphic to $L(G^{e},\emptyset)\ba e$, where
$G^{e}\in\mcal{G}$ has $p$ parallel pairs, $c$ loops, and
$l+1$ thin edges, one of which is $e$.
Because of these arguments, it now suffices to show that
Category\dash (A), (B), and (C) matroids are in $\mcal{S}_{k}$.

Let $M$ be an $n$\dash element Category\dash (C) matroid.
Note that $M$ has at most one matroid loop.
Construct the two-vertex graph $G^{e}$ with $n-1$ loops
and two non-loop edges, one of which is $e$.
If $M$ has a loop, then set $\mcal{B}^{e}$ to contain the
unique Hamiltonian cycle of $G^{e}$, and otherwise make
$\mcal{B}^{e}$ empty.
Then $M = L(G^{e},\mcal{B}^{e})/e$.
Next we let $M=L(G,\mcal{B})$ be a Category\dash (B) matroid.
If $G$ has at most one non-loop, then $M$ is the union
of a coloop and a parallel class.
In this case it is easy to see that $M$ is a minor of a
Category\dash (A) matroid.
Therefore we assume that $G$ has at least two non-loop edges.
Since $G$ has at most four such edges, \mcal{B}
contains at most two Hamiltonian cycles.
We construct $G^{e}$, a three-vertex graph in \mcal{G}.
We set the number of non-loop edges in $G^{e}$ to be one
more than the number of non-loops in $G$, and we
make $e$ a thin edge of $G^{e}$.
Let the number of loops in $G^{e}$ be equal to the number of
loops in $G$.
We can find a linear class $\mcal{B}^{e}$ of Hamiltonian cycles
in $G^{e}$ so that $|\mcal{B}^{e}|=|\mcal{B}|$.
Now $M$ is isomorphic to $L(G^{e},\mcal{B}^{e})/e$,
so we have reduced the proof to showing that
every Category\dash (A) matroid is in $\mcal{S}_{k}$.

Let $M=L(G,\mcal{B})$ be a Category\dash (A) matroid.
Let the parallel pairs of $G$ be
$\{a_{1},b_{1}\},\ldots, \{a_{t},b_{t}\}$, let the loops be
$\{c_{1},\ldots, c_{p}\}$, and let the thin edges be
$\{d_{1},\ldots, d_{s}\}$.
We construct $G^{+}$ isomorphic to $\Delta_{t+p+s}$ with parallel pairs
$\{a_{1},b_{1}\},\ldots, \{a_{t},b_{t}\}$,
$\{c_{1},x_{1}\},\ldots, \{c_{p},x_{p}\}$, and
$\{d_{1},y_{1}\},\ldots, \{d_{s},y_{s}\}$.
Let $C_{1},\ldots, C_{q}$ be the cycles in \mcal{B}.
For each $C_{i}$, let $C_{i}^{+}$ be the Hamiltonian
cycle of $G^{+}$ containing all of
$x_{1},\ldots, x_{p}$ and $d_{1},\ldots, d_{s}$,
and such that $C_{i}^{+}$ intersects $\{a_{j},b_{j}\}$ in the same
edge as $C_{i}$ for each $j$.
It is clear that $\mcal{B}^{+}$ is a linear class.
Furthermore $M$ is isomorphic to
\[
L(G^{+},\mcal{B}^{+})/\{x_{1},\ldots, x_{p}\}\ba\{y_{1},\ldots, y_{s}\},
\]
so $M$ is in $\mcal{S}_{k}$, and the proof of the
\namecref{safety} is complete.
\end{proof}

\begin{proposition}
\label{galaxy}
Let $M=L(G,\mcal{B})$ be a Category\dash(A) matroid, where $G$ has at
least five vertices.
If $C$ is a circuit-hyperplane in $M$, then $G[C]$ is a cycle in \mcal{B}.
\end{proposition}

\begin{proof}
This follows very easily from \Cref{charge}.
We note that the constraint $|V(G)|\geq 5$ is necessary, for if $G$
has four vertices, then a pair of parallel pairs in $G$ may form a
circuit-hyperplane of $M$.
\end{proof}

We use the symbol $\graph$ to denote the graph obtained from
a three-vertex cycle by adding a single parallel edge.

\begin{proposition}
\label{camera}
The following matroids are not in \mcal{S}.
\begin{enumerate}[label=\textup{(\roman*)}]
\item $U_{0,1}\oplus U_{1,1}\oplus U_{1,3}$,
\item $U_{0,1}\oplus U_{1,1}\oplus U_{2,3}$,
\item $U_{0,1}\oplus U_{2,4}$,
\item $U_{1,1}\oplus U_{2,4}$,
\item $U_{1,2}\oplus M(\graph)$.
\end{enumerate}
\end{proposition}

\begin{proof}
By virtue of \Cref{attack}, we need only prove that the matroids in
(i), (iii), and (v) are not in \mcal{S}.
If a matroid belongs to \mcal{S}, then it belongs to
$\mcal{S}_{k}$ for some value of $k$.
Therefore we can apply \Cref{safety}.
Note that if a matroid in $\mcal{S}_{k}$ has a loop, then it is
Category\dash (C), (D), (E), or (F).
If a Category\dash (C), (D), or (E) matroid has a loop and a coloop,
then every element is a loop or a coloop.
(For example, the only way a Categry\dash (D) matroid can have a
coloop is if it is the cycle matroid of a graph with a thin edge
joining two vertices.)
Category\dash (F) matroids have no components of size three.
In any case, we see that
 $U_{0,1}\oplus U_{1,1}\oplus U_{1,3}$ is not in $\mcal{S}_{k}$.
Category\dash (C) matroids do not have rank two,
Category\dash (D) and (E) matroids obviously have no $U_{2,4}$\dash minor
as they are graphic or cographic, and nor do Category\dash (F) matroids.
Therefore $U_{0,1}\oplus U_{2,4}$ is not in $\mcal{S}_{k}$.

Finally, Category\dash (A), (B), or (C) matroids have at most
one non-trivial parallel class, so they cannot be isomorphic to
$U_{1,2}\oplus M(\graph)$.
If $U_{1,2}\oplus M(\graph)$ is Category\dash (D), then it is isomorphic to
$M(G)$ for some $G\in \mcal{G}$ with four vertices.
But any such matroid is connected up to loops, so
$U_{1,2}\oplus M(\graph)$ is not Category\dash (D).
Duality tells us it is not Category\dash (E) either, and it
certainly has a connected component with more than two
elements so it is not Category\dash (F).
\end{proof}

\begin{lemma}
\label{output}
Let $M$ be an excluded minor for \mcal{S} such that $r(M),r^{*}(M)>2$
and $M$ is not isomorphic to $U_{1,2}\oplus M(\graph)$.
Then $M$ is simple.
\end{lemma}

\begin{proof}
We start with the following claim.

\begin{claim}
\label{patrol}
Let $M$ be an excluded minor for \mcal{S} such that $r(M),r^{*}(M)>2$.
Then $M$ is loopless.
\end{claim}

\begin{proof}
Assume the contrary, and let $e$ be a loop of $M$.
If every connected component of $M\ba e$ has size at most two, then the
same statement applies to $M$, so \Cref{safety} now implies that $M$ is in
\mcal{S}, a contradiction.
Therefore we let $N$ be a component of $M\ba e$ with at least three elements.
It is an easy exercise to see that $N$ has a minor isomorphic to either
$U_{1,3}$ or $U_{2,3}$ (see \cite[Exercise 9 Chapter 4]{Oxl11}).
If there is another component of $M\ba e$ with rank at least one, then
$M$ has a minor isomorphic to
$U_{0,1}\oplus U_{1,1}\oplus U_{1,3}$ or $U_{0,1}\oplus U_{1,1}\oplus U_{2,3}$.
In this case, \Cref{camera} implies that $M$ must be isomorphic to one of these
two matroids, so
$M$ has rank or corank equal to two, a contradiction to the hypotheses.
Thus every component of $M\ba e$ other than $N$ is a loop.
Since $r(M)\geq 3$, we deduce that $r(N)\geq 3$.

\Cref{camera} implies that if $N$ has an $U_{2,4}$\dash minor, then
$M$ is isomorphic to $U_{0,1}\oplus U_{2,4}$, which is not possible.
Therefore $N$ is binary.
Furthermore, $N$ cannot have a minor isomorphic to $M(K_{4})$, or else
$M$ has a minor isomorphic to $U_{0,1}\oplus U_{1,1}\oplus U_{2,3}$.
Therefore $N$ is graphic (see \cite[Theorem 10.4.8]{Oxl11}).
We let $G$ be a graph such that $N= M(G)$.
As $r(N)\geq 3$, it follows that $G$ has at least three vertices, and
since $N$ is a connected component, it follows that $G$ is $2$-connected.
This implies that $G$ has a cycle, $C$, containing at least three vertices.
Assume that there is an edge, $x$, of $G$ such that
$x$ is neither in $C$, nor parallel to an edge of $C$.
Then there exists a cycle of $G$ with at least three vertices, and an edge
that is not in the span of that cycle.
Thus $N$ has a minor isomorphic to $U_{1,1}\oplus U_{2,3}$, and hence
$M$ has a minor isomorphic to $U_{0,1}\oplus U_{1,1}\oplus U_{2,3}$.
This leads to a contradiction, so every edge of $G$ is either in $C$, or
parallel to an edge in $C$.
If $G$ contains a parallel class of size at least three, then
$N$ has a minor isomorphic to $U_{1,1}\oplus U_{1,3}$, which
again leads to a contradiction.
Therefore $G$ is obtained from a cycle of at least three vertices
by adding loops and parallel edges in such a way that any parallel class
has size one or two.
Now it follows that $M$ is the cycle matroid of a graph in \mcal{G}, and
hence \Cref{safety} implies that $M$ is in \mcal{S}, a contradiction.
\end{proof}

Let $M$ be an excluded minor \mcal{S} satisfying $r(M),r^{*}(M)>2$
and assume that $M$ is not isomorphic to $U_{1,2}\oplus M(\graph)$.
Then $M$ has no loops by \Cref{patrol}.
Since \Cref{patrol} also applies to $M^{*}$, we
deduce that $M$ has no coloops.
Assume that $M$ has at least one parallel pair,
and let $\{x,y\}$ be such a pair.

\begin{claim}
\label{ticker}
$M/x$ has a connected component containing at least three elements.
\end{claim}

\begin{proof}
Assume for a contradiction that every connected component of $M/x$
has size one or two.
Then every connected component of $M/x$ is a loop, or a $2$\dash element circuit,
since $M$ has no coloops.
Let $L$ be the set of loops of $M/x$.
Then $L\cup x$ is a parallel class of $M$, since $M$ is loopless.
Let $C_{1},\ldots, C_{s}$ be the $2$\dash element components of
$M/x$ that are not circuits in $M$, and let $D_{1},\ldots, D_{t}$ be the
$2$\dash element components that are circuits of $M$.
Note that $s+t\geq 2$, since $r(M)>2$ implies $r(M/x)\geq 2$.

If $t=0$, then $M$ is isomorphic to $M^{*}(G)$, where $G$ is
obtained from a cycle of length $s + |L| + 1$ by replacing $s$
of the edges with parallel pairs.
Thus $G$ is in \mcal{G}, and \Cref{safety} implies that $M$ is in
\mcal{S}, a contradiction.
Therefore $t>0$.
Assume that $s=0$, so that $t\geq 2$.
The connected components of $M$ are now
$D_{1},\ldots, D_{t}$ and $L\cup x$.
Thus $|L\cup x| >2$, or else every component of $M$ has
size at most two, which is a contradiction as $M$ is not in \mcal{S}.
We contract an element from $D_{1}$, so that the other element of
$D_{1}$ is now a loop.
We choose a single element from $D_{2}$, and three elements from
$L\cup x$.
Now we see that $M$ has a proper minor isomorphic to
$U_{0,1}\oplus U_{1,1}\oplus U_{1,3}$, which contradicts \Cref{camera}.
Therefore $s$ and $t$ are both positive.
We let $l$ be an element from $L$.
The restriction of $M$ to $C_{1}\cup D_{1}\cup \{x,l\}$ is isomorphic to
$U_{1,2}\oplus M(\graph)$.
So $M$ is isomorphic to this matroid, a contradiction.
This contradiction completes the proof of \Cref{ticker}.
\end{proof}

As $M$ is an excluded minor, we know that $M/x$ is a member of \mcal{S}.
Therefore it satisfies one of the statements in \Cref{safety}.
\Cref{ticker} shows that $M/x$ is not Category\dash(F).
Category\dash(A) and (B) matroids do not have loops, and $M/x$ contains the
loop $y$, so it belongs to neither of these categories.
As the rank of $M/x$ is at least two, it is not Category\dash(C).
If $M/x=M^{*}(G)$ for some $G\in \mcal{G}$, then $G$ has an
isthmus, since $M/x$ has a loop.
This is only possible if $G$ is obtained from $K_{2}$ by adding loops.
In this case $M/x$ has at least two coloops (since $r(M/x)\geq 2$).
But this is impossible, as $M$ has no coloops.
Therefore $M/x$ is not Category\dash(E), so it must be Category\dash(D).
Let $G\in \mcal{G}$ be chosen so that $M/x = M(G)$, and let $L$ be
the set of loops of $G$.
Note that $y$ is in $L$, and that $G$ has at least three vertices, since $r(M)> 2$.

We let $G^{x}$ be the graph obtained from $G$ by adding $x$ as a loop.
Let \mcal{B} be the collection of cycles of $G$ that correspond to circuits of $M$.
Since $M$ is loopless, \mcal{B} contains no loop.
\Cref{ritual} tells us that $M = L(G^{x},\mcal{B})$.
If \mcal{B} contains only Hamiltonian cycles, then $M$ is in \mcal{S}, a
contradiction.
Therefore \mcal{B} must contain a cycle with two edges, $a$ and $b$.
Hence $\{a,b\}$ is a circuit of $M$.
Assume that there is a two-edge cycle that is not in \mcal{B}
and let those edges be $c$ and $d$.
Then the restriction of $M$ to $\{a,b,c,d,x,y\}$ is isomorphic to
$U_{1,2}\oplus M(\graph)$.
This implies that $M$ is isomorphic to $U_{1,2}\oplus M(\graph)$,
which is a contradiction.
We conclude that \mcal{B} contains every two-edge cycle of $G^{x}$.

Assume that \mcal{B} contains no Hamiltonian cycles.
The only circuits of $M$ are pairs in $L\cup x$, parallel pairs in $G$,
and the union of an element in $L\cup x$ with the edge-set of
a Hamiltonian cycle.
In other words, $M$ is obtained from a circuit by adding parallel elements.
If $|L\cap x|\leq 2$, then $M$ is the cycle matroid of a graph in
\mcal{G}, a contradiction.
Therefore $|L\cup x|\geq 3$.
If $G$ has no parallel pairs, then $M$ is the lift matroid of
a graph obtained from a cycle by adding loops
(where no cycle is balanced).
In this case $M$ is a Category\dash (A) matroid, which is
impossible.
Therefore $G$ contains a parallel pair of edges, $a$ and $b$.
We contract $a$, so that $b$ is a loop, and then
select $b$, three elements from $L\cup x$, and a single element
not in $L\cup \{x,a,b\}$ (this element exists because $r(M)\geq 3$).
This shows that $M$ has a proper minor isomorphic to
$U_{0,1}\oplus U_{1,1}\oplus U_{1,3}$, a contradiction.

Now we know that \mcal{B} contains a Hamiltonian cycle, $C_{1}$.
Assume there is a Hamiltonian cycle, $C_{2}$ that is not in \mcal{B},
and assume that we have chosen $C_{2}$ so that it has as many edges
in common with $C_{1}$ as possible.
Let $\{e,f\}$ be a parallel pair of $G^{x}$ such that $e$ is in $C_{1}$ and $f$
is in $C_{2}$.
Let $C_{3}$ be the Hamiltonian cycle obtained from $C_{2}$ by removing $f$
and replacing it with $e$.
Then $C_{3}$ is in \mcal{B} by our choice of $C_{2}$.
The theta-subgraph obtained from $C_{3}$ by adding $f$ contains
$C_{2}$, $C_{3}$, and $\{e,f\}$, and exactly two of these cycles are in
\mcal{B}, which contradicts the fact that \mcal{B} is a linear class.
We conclude that every Hamiltonian cycle is in \mcal{B}.
This means that the only cycles of $G^{x}$ not in \mcal{B} are the loops.
It follows that $M$ is the direct sum of the cycle matroid $M(G\ba L)$
and the parallel class $L\cup x$.
But $G\ba L$ has a minor isomorphic to $\graph$, and hence
$M$ has a minor isomorphic to $U_{1,2}\oplus M(\graph)$, and this leads to a
contradiction that completes the proof of \Cref{output}.
\end{proof}

\begin{lemma}
\label{cheese}
Let $k$ be a non-negative integer.
There exists an integer, $N_{k}$, with the following property:
if $M$ is an excluded minor for $\mcal{S}_{k}$ with $|E(M)|>N_{k}$,
then $|E(M)|$ is even.
\end{lemma}

\begin{proof}
We start by noting that matroids of rank at most two
are well-quasi-ordered.
This is not difficult to prove directly, but it also follows
from \cite{HO10}, since $U_{3,3}$ is the sole excluded minor
for the class of matroids with rank at most two, and
the main theorem in \cite{HO10} implies that the class
produced by excluding $U_{3,3}$ does not
contain any infinite antichains.
So there are only finitely many excluded minors for
$\mcal{S}_{k}$ with rank (or corank, by duality) at most two.

We let $\mcal{S}'$ stand for the class of
Category\dash(D), (E), or (F) matroids.
By referring to the proof of \Cref{safety},
we can easily verify that $\mcal{S}'$ is a minor-closed class.
Note that all matroids in $\mcal{S}'$ are graphic
(since graphs in \mcal{G} are planar).
Therefore any excluded minor for $\mcal{S}'$ is either
an excluded minor for the class of graphic matroids,
or it is itself graphic.
There are only five excluded minors for the class
of graphic matroids \cite{Tut58}.
The class of graphic matroids is well-quasi-ordered
by the famous result of Robertson and Seymour \cite{RS04},
so there are only finitely many graphic excluded minors for $\mcal{S}'$.
Thus $\mcal{S}'$ has finitely many excluded minors.
(With some extra effort, we could find the excluded minors for
$\mcal{S}'$ directly, in which case we would not have to rely on
Robertson and Seymour's result.)

These arguments show that we can choose 
$N_{k}$ so that it satisfies $N_{k}\geq 12$, and also
$N_{k}\geq |E(M)|$ whenever $M$ is an excluded minor for $\mcal{S}'$
or an excluded minor for $\mcal{S}_{k}$ with rank or corank at most two.
Now, if $M$ is an excluded minor for $\mcal{S}_{k}$ such that
$|E(M)|>N_{k}$, then $r(M),r^{*}(M)\geq 3$, and $M$ is not an
excluded minor for $\mcal{S}'$.

\begin{claim}
\label{orange}
Let $M$ be an excluded minor for $\mcal{S}_{k}$.
Assume $|E(M)|$ is odd and larger than $N_{k}$.
Then $M$ is not in \mcal{S}.
\end{claim}

\begin{proof}
Assume otherwise, so that $M$ belongs to $\mcal{S}_{k'}$ for some $k'>k$.
We apply \Cref{safety} to $M$.
Since $M$ is not in $\mcal{S}'$,
it is not Category\dash(D), (E), or (F).
It is also not Category\dash(B) or (C), as $r(M)\geq 3$.
Therefore $M$ is Category\dash(A), so $M=L(G,\mcal{B})$, where
$G\in\mcal{G}$ has at least three vertices, and \mcal{B} is a linear class
of at most $k'$ Hamiltonian cycles.
As $M$ is not in $\mcal{S}_{k}$, it follows that $k<|\mcal{B}|\leq k'$.
Furthermore,  $|E(M)|$ is odd, so the edge-set of $G$ is not a union of
parallel pairs.
Hence $G$ has either a loop or a thin edge.
Note that $M$ has no loops or coloops.

Assume that $G$ has at most four vertices.
Since $|E(M)|>N_{k}\geq 12$, it follows that $G$ has 
at least five loops.
Thus we can let $P$ be a parallel class of $M$ 
satisfying $|P|\geq 5$.
Let $x$ and $y$ be distinct elements in $P$.
We apply \Cref{safety} to $M\ba x$, which is in $\mcal{S}_{k}$.
Since $M\ba x$ has a parallel class of size at least four, it is not Category\dash(D) or (F).
Furthermore, $r(M)\geq 3$ implies $r(M\ba x)\geq 3$, so
it is not Category\dash(B) or (C).
Hence $M\ba x$ is Category\dash(A) or (E).
If $M\ba x$ is Category\dash(E), then $M\ba x = M^{*}(G_{x})$ for some graph
$G_{x}\in\mcal{G}$, and the elements in $P-x$ are thin edges of
$G_{x}$.
Let $G_{x}^{+}$ be obtained from $G_{x}$ by subdividing $y$, and
naming the two new edges $x$ and $y$.
Then $M^{*}(G_{x}^{+})$ is obtained from $M^{*}(G_{x})$ by placing
$x$ parallel to $y$.
It follows that $M = M^{*}(G_{x}^{+})$, and hence $M$ is in $\mcal{S}'$,
a contradiction.

Therefore $M\ba x$ is Category\dash(A), so it is equal to $L(G_{x},\mcal{B}_{x})$, where
$G_{x}\in \mcal{G}$ contains at least three vertices, and $\mcal{B}_{x}$ contains at
most $k$ Hamiltonian cycles of $G_{x}$.
The only parallel pairs in $L(G_{x},\mcal{B}_{x})$ arise from loops of $G_{x}$.
We deduce that the elements of $P-x$ are loops in $G_{x}$.
We obtain $G_{x}^{+}$ from $G_{x}$ by adding $x$ as a loop incident with an arbitrary
vertex.
Then $L(G_{x}^{+},\mcal{B}_{x})$ is obtained by adding $x$ parallel to $y$, so
$L(G_{x}^{+},\mcal{B}_{x})=M$.
This implies $M$ is in $\mcal{S}_{k}$, a contradiction.
Therefore $G$ has at least five vertices and
hence $r(M)\geq 5$.

Recall that $G$ has either a loop or a thin edge.
Assume that $G$ has a loop, $x$.
Note that $M\ba x = L(G\ba x,\mcal{B})$.
Since $M\ba x$ is in $\mcal{S}_{k}$, we apply \Cref{safety}.
Because $r(M\ba x)\geq 5$, it follows that $M\ba x$ is not Category\dash(B) or (C).
Note that \mcal{B} is not empty, since it contains more than $k$ Hamiltonian cycles.
Any cycle in \mcal{B} corresponds to a circuit-hyperplane in $M\ba x$,
which necessarily has at least five elements.
Therefore $M\ba x$ is not Category\dash(F).
The only cycle matroids of graphs in \mcal{G}
that have rank at least five and a circuit-hyperplane are
isomorphic to $U_{n-1,n}$.
But $M\ba x$ is not isomorphic to $U_{n-1,n}$, because
$r^{*}(M\ba x)\geq 2$.
We can conclude that $M\ba x$ is not Category\dash(D).
A simple analysis shows that a Category\dash(E) matroid
with rank at least five has no circuit-hyperplane, so
$M\ba x$ is not Category\dash(E).
The only remaining possibility is that $M\ba x$ is Category\dash(A).
Therefore $M\ba x = L(G_{x},\mcal{B}_{x})$, where
$G_{x}\in \mcal{G}$ has at least three vertices, and $\mcal{B}_{x}$ contains at
most $k$ Hamiltonian cycles.
Note that $G_{x}$ has at least five vertices, as $r(M\ba x)\geq 5$.
\Cref{galaxy} implies that $M\ba x$ has at most $k$ circuit-hyperplanes,
which is impossible because \mcal{B} contains at least $k+1$ cycles,
and thus $M$ has at least $k+1$ circuit-hyperplanes that avoid $x$.
Thus $G$ has no loop.
By an earlier conclusion, we can let $x$ be a thin edge.

As $G$ has no loops and at least $13$ edges, we can now see that
$G$ has at least six vertices, so $r(M)\geq 6$.
We consider the matroid $M/x$, which is in $\mcal{S}_{k}$.
As in the previous paragraph, we can argue that
$M/x = L(G_{x},\mcal{B}_{x})$, where $G_{x}\in\mcal{G}$
has at least five vertices, and $\mcal{B}_{x}$ contains at most $k$
cycles.
This implies that $M/x$ has at most $k$ circuit-hyperplanes.
But this is impossible, as \mcal{B} contains at least $k+1$ cycles,
and each corresponds to a circuit-hyperplane of $M$ that contains $x$.
\end{proof}

Now, whenever $M$ is an excluded minor for $\mcal{S}_{k}$
such that $|E(M)|$ is odd and larger than $N_{k}$,
\Cref{orange} tells us that it is an excluded minor for \mcal{S}.
Since $r(M),r^{*}(M)\geq 3$ and $|E(M)|>12$, 
\Cref{output} implies that $M$ is simple.
As $M^{*}$ is also an excluded minor for \mcal{S}, we can deduce that
$M$ is cosimple.

\begin{claim}
\label{stream}
Let $M$ be an excluded minor for $\mcal{S}_{k}$.
Assume $|E(M)|$ is odd and larger than $N_{k}$.
If, for some $e\in E(M)$, we have $M\ba e = L(G_{e},\mcal{B}_{e})$,
where $G_{e}\in\mcal{G}$ has at least three vertices and $\mcal{B}_{e}$ is a
linear class of at most $k$ Hamiltonian cycles, then $G_{e}$ has no loops, and at least
six vertices.
\end{claim}

\begin{proof}
Assume $x$ is a loop in $G_{e}$.
Note that $M\ba e / x = M(G_{e}\ba x)$, so that
$M\ba e/ x$ is Category\dash(D).
As $M$ is simple, $M\ba e/x$ has no matroid loops.
Therefore $x$ is the unique loop in $G_{e}$.
We will apply \Cref{safety} to $M/x$, which is in $\mcal{S}_{k}$.
Our aim is to deduce that $M/x$ too is Category\dash(D).

Since $G_{e}$ has exactly one loop, and at least twelve edges,
it follows that it has more than five vertices.
Therefore $r(M)> 5$.
This immediately rules out the cases where $M/x$ is
Category\dash(B) or (C).
Furthermore, if we let $C$ be the edge-set of any Hamiltonian cycle in $G_{e}$,
then either $C$ is a circuit-hyperplane of $M\ba e$, or $C\cup x$ is a circuit.
In any case, $M/x$ has a circuit of more than five elements, so
it is not Category\dash(F).

We note that Category\dash(A) and (E) matroids have
at most one parallel class.
So if $M/x$ belongs to either of these categories, then
$M/x\ba e = M(G_{e}\ba x)$ too has at most one parallel class.
This implies that $G_{e}$ has at most one parallel pair.
But in this case, $G_{e}$ comprises a cycle, a loop, and
at most one parallel edge.
This implies that $r^{*}(M\ba e) \leq 1$, so $r^{*}(M)\leq 2$,
a contradiction.
Therefore we can conclude that $M/x$ is Category\dash(D), exactly as
we wanted.

Now let $G_{x}\in\mcal{G}$ be chosen so that $M/x = M(G_{x})$.
Let $G_{x}^{+}$ be constructed from $G_{x}$ by adding the
loop $x$ to an arbitrary vertex.
\Cref{ritual} asserts that $M = L(G_{x}^{+},\mcal{B})$, for some
linear class \mcal{B} of cycles in $G_{x}^{+}$.
But \mcal{B} cannot contain a cycle with one or
two edges, for $M$ is simple.
Therefore \mcal{B} contains only Hamiltonian cycles of $G_{x}^{+}$.
This demonstrates that $M$ is in \mcal{S}, which is
impossible, according to \Cref{orange}.
Therefore $G_{e}$ has no loops, and as it has at least twelve
edges, it has at least six vertices.
\end{proof}

\begin{claim}
\label{weight}
Let $M$ be an excluded minor for $\mcal{S}_{k}$.
Assume $|E(M)|$ is odd and larger than $N_{k}$.
Assume also that $M\ba e = L(G_{e},\mcal{B}_{e})$ for some $e\in E(M)$,
where $G_{e}\in\mcal{G}$ has at least three vertices and $\mcal{B}_{e}$ is a
linear class of at most $k$ Hamiltonian cycles.
Then $G_{e}$ contains at least three parallel pairs, and
if $\{a,b\}$ is a parallel pair in $G_{e}$, then $\{e, a,b\}$ is a
circuit of $M$.
\end{claim}

\begin{proof}
Let the parallel pairs in $G_{e}$ be
$\{a_{1},b_{1}\},\ldots, \{a_{t},b_{t}\}$.
Assume that $t\leq 2$.
Since $G_{e}$ has no loops by \Cref{stream}, it follows that
$r^{*}(M\ba e)\leq 1$, which is impossible.
Hence $t\geq 3$.
Since the numbering of the parallel pairs is arbitrary, we can finish the proof by
showing that $\{e,a_{1},b_{1}\}$ is a circuit of $M$.

Note that $M\ba e / a_{1} = L(G_{e}/a_{1},\mcal{B}_{e}/a_{1})$,
where $\mcal{B}_{e}/a_{1}$ is obtained from $\mcal{B}_{e}$
by removing the Hamiltonian cycles that do not contain
$a_{1}$, and then contracting $a_{1}$ from each of the remaining
cycles.
As $b_{1}$ is a loop in $G_{e}/a_{1}$, it follows that
$\{b_{1},a_{i},b_{i}\}$ is a triangle in $M\ba e/a_{1}$ for each
$i\geq 1$, and is thus a triangle in $M/a_{1}$ and 
a triad in $M^{*}\ba a_{1}$.

We apply \Cref{safety} to $M^{*}\ba a_{1}$.
Because it has triads, it is not Category\dash(F).
Since $r(M^{*}\ba a_{1}) = r^{*}(M)\geq 3$, it
follows that $M^{*}\ba a_{1}$ is not Category\dash(B) or (C).
Note that $M^{*}\ba a_{1}$ has no parallel pairs and no loops,
as $M$ is cosimple.
So if $M^{*}\ba a_{1}$ is Category\dash(D), then it is a circuit, which
is impossible as its corank is at least two.
Now assume that $M^{*}\ba a_{1}$ is Category\dash(E), so that
$M^{*}\ba a_{1}= M^{*}(G_{a_{1}})$ for some $G_{a_{1}}\in\mcal{G}$.
Then $G_{a_{1}}$ has no loops, since $M^{*}\ba a_{1}$ has no coloops.
Furthermore, $M^{*}\ba a_{1}$ is simple, so
$G_{a_{1}}$ has at most one thin edge.
But $G_{a_{1}}$ also has an even number of edges, so it follows that
$G_{a_{1}}$ is isomorphic to $\Delta_{r}$ where $r = \tfrac{1}{2}(|E(M)|-1)$.
But then $M^{*}\ba a_{1}$ has no triads, which is impossible.
We are forced to conclude that $M^{*}\ba a_{1}$ is
Category\dash(A).

Now we choose $G_{a_{1}}\in \mcal{G}$ and a linear class
$\mcal{B}_{a_{1}}$ of at most $k$ Hamiltonian cycles so that
$M^{*}\ba a_{1} = L(G_{a_{1}},\mcal{B}_{a_{1}})$.
By \Cref{stream} we see that $G_{a_{1}}$ has no loops and
at least six vertices.
We observe that the only triads of $M^{*}\ba a_{1}$ consist
of a parallel pair in $G_{a_{1}}$ along with a thin edge.
Since $\{b_{1},a_{i},b_{i}\}$ is a triad in $M^{*}\ba a_{1}$ for each
$i\in \{2,\ldots, t\}$, we see that $b_{1}$ must be a thin
edge of $G_{a_{1}}$, and each $\{a_{i},b_{i}\}$ is a parallel pair.
As $G_{a_{1}}$ has an even number of edges, we let $x$ be
another thin edge, distinct from $b_{1}$.
Then $\{b_{1}, x\}$ is a cocircuit of $M^{*}\ba a_{1}$, and hence
a circuit in $M/a_{1}$.
Since $M$ is simple, this implies that $\{x,a_{1},b_{1}\}$ is a circuit.
If $x=e$, then there is nothing left to prove, so we assume that $x\ne e$.
Therefore $\{x,a_{1},b_{1}\}$ is a circuit of $M\ba e$.
But this is impossible, as $G_{e}$ has no loops, and since $G_{e}$
has at least six vertices, it follows that $M\ba e = L(G_{e},\mcal{B}_{e})$
has no triangles.
This completes the proof of the \namecref{weight}.
\end{proof}

Now we let $M$ be an excluded minor for $\mcal{S}_{k}$
such that $|E(M)|$ is larger than $N_{k}$ and odd.
We recall that $M$ is simple and cosimple.
As $M$ is not an excluded minor for $\mcal{S}'$,
there is an element $e\in E(M)$ such that either $M\ba e$ or $M/e$
is in $\mcal{S}_{k}$ but not in $\mcal{S}'$.
By duality, we can assume that $M\ba e$ is in $\mcal{S}_{k}-\mcal{S}'$.
We apply \Cref{safety}.
From $r(M)\geq 3$ we deduce that $M\ba e$ is not Category\dash(B) or (C), and
as it is not in $\mcal{S}'$, $M\ba e$ must be Category\dash(A).
Choose $G_{e}\in\mcal{G}$ and $\mcal{B}_{e}$, a linear class of
at most $k$ Hamiltonian cycles in $G_{e}$, such that $M\ba e = L(G_{e},\mcal{B}_{e})$.
From Claims \ref{stream} and \ref{weight}, we know that $G_{e}$ has no loops,
at least six vertices, and least three parallel pairs.
Thus $r(M)\geq 6$.
Let $\{a_{1},b_{1}\},\ldots, \{a_{t},b_{t}\}$ be the parallel pairs of edges.
Then $\{e,a_{i},b_{i}\}$ is a triangle of $M$ for each $i$.

Now we apply \Cref{safety} to $M\ba a_{1}$.
As this matroid contains triangles it is not Category\dash(F).
The inequality $r(M)\geq 6$ rules out Categories-(B) and (C).
If $M\ba a_{1}$ is Category\dash(D), then it must be a circuit,
for these are the only simple Category\dash(D) matroids.
But this would contradict $r^{*}(M)\geq 3$.
If $M\ba a_{1}$ is Category\dash(E), then
$M\ba a_{1}=M^{*}(G_{a_{1}})$ for some $G_{a_{1}}\in\mcal{G}$.
Because $M\ba a_{1}$ has no coloops, $G_{a_{1}}$ has no loops.
If it contains a thin edge, then it contains two thin
edges, as the number of edges in $G_{a_{1}}$ is even.
This implies $M\ba a_{1}$ contains a parallel pair, which is impossible.
So $G_{a_{1}}$ is $\Delta_{r}$, where $r=\tfrac{1}{2}(|E(M)|-1)$.
But then $M^{*}(G_{a_{1}})$ contains no triangles, and this
is impossible since $\{e,a_{2},b_{2}\}$ is a triangle of $M\ba a_{1}$.
Therefore $M\ba a_{1}$ is Category\dash(A).
Now we can apply \Cref{weight} to $M\ba a_{1}$.
It tells us that $a_{1}$ is in at least three triangles of $M$,
and that the intersection of any pair of these triangles
is $\{a_{1}\}$.
From this it follows that $a_{1}$ is in a triangle of
$M\ba e=L(G_{e},\mcal{B}_{e})$, which is impossible
as $G_{e}$ has at least six vertices, and no loops.
Now the proof of \Cref{cheese} is complete.
\end{proof}

We are positive that there are only finitely many excluded minors
for \mcal{S}.
But we do not require this for our main results, so we leave it
as an open problem.

\begin{problem}
Prove that \mcal{S} has only finitely many excluded minors,
and describe all of them.
\end{problem}

From this point onwards our strategy in proving \Cref{memory} is similar
to that used in the proof of \Cref{valley}.
We start by recalling \Cref{crutch}:
if $\mcal{C}=(C_{1},\ldots, C_{k})$ is a sequence of subsets of
the set $E$, then for any $I\subseteq \{1,\ldots, k\}$, we use
$\mcal{C}(I)$ to denote the set
$\{e\in E\colon e\in C_{i} \Leftrightarrow i\in I\}$.
The function $\psi_{\mcal{C}}$ takes each $I$ to $|\mcal{C}(I)|$.
If $\mcal{C}=(C_{1},\ldots, C_{k})$ is a sequence of sets, then
we define $\trun(\mcal{C})$ to be the derived sequence
$(C_{1}\cap C_{k},\ldots, C_{k-1}\cap C_{k})$.

Let $k$ and $r$ be integers
satisfying $k\geq 2$ and $r\geq 5$.
The multivalued function, $\mcal{R}_{k}^{r}$, has as its
domain the set of rank\dash $r$ Category\dash (A)
matroids with exactly $k$ circuit-hyperplanes.
The codomain is the set of functions from $\mcal{P}(\{1,\ldots, k-1\})$
to $\mbb{Z}_{\geq 0}$.
Let $M$ be a matroid in the domain, and let $\psi$ be a function
in the codomain.
The ordered pair $(M,\psi)$ belongs to $\mcal{R}_{k}^{r}$
if and only if there is an ordering $\mcal{C}=(C_{1},\ldots, C_{k})$ of the
circuit-hyperplanes in $M$ such that
$\psi$ is equal to $\psi_{\trun(\mcal{C})}$.
In this case, $\psi$ takes $I\subseteq \{1,\ldots, k-1\}$ to the
number of elements in $C_{k}$ that are in every
$C_{i}$ for $i\in I$, and in no $C_{i}$ for $i\notin I$.
Furthermore,
\[
\sum_{I\subseteq \{1,\ldots, k-1\}} \psi(I) = r.
\]
Note that the image of $M$ under $\mcal{R}_{k}^{r}$ has
cardinality at most $k!$.

A Category\dash (A) matroid $M=L(G,\mcal{B})$ with rank at
least five has at most one non-trivial parallel class
(comprising the loops of $G$), and at most one non-trivial series
class (comprising the thin edges).
Any triangle of $M$ comprises a loop of $G$ and a parallel pair,
and any triad comprises a parallel pair along with a thin edge.
Moreover, the circuit-hyperplanes of $M$ correspond exactly to the
cycles in \mcal{B}, as we noted in \Cref{galaxy}.

Let $M$ be a Category\dash (A) matroid.
If $M$ has no triangle, we set $p(M)$ to be zero, and otherwise we set it
to be the maximum size of a parallel class in $M$.
Similarly, if $M$ has no triad, we set $s(M)$ to be zero, and otherwise we set
it to be the largest size of a series class.

\begin{proposition}
\label{glance}
Let $k$ and $r$ be integers satisfying $k\geq 2$ and $r\geq 5$.
Let $M$ and $N$ be rank\dash $r$ Category\dash \textup{(A)} matroids
with exactly $k$ circuit-hyperplanes.
Then $M$ and $N$ are isomorphic if and only if
\begin{enumerate}[label=\textup{(\roman*)}]
\item $p(M) = p(N)$,
\item $s(M) = s(N)$, and
\item $\mcal{R}_{k}^{r}(M)\cap \mcal{R}_{k}^{r}(N)\ne \emptyset$.
\end{enumerate}
\end{proposition}

\begin{proof}
Let $\rho$ be an isomorphism from $M$ to $N$.
The existence of $\rho$ clearly means that $p(M)=p(N)$ and
$s(M) = s(N)$.
Let $(C_{1},\ldots, C_{k})$ be an ordering of the
circuit-hyperplanes in $M$.
Then $\rho(\mcal{C})=(\rho(C_{1}),\ldots, \rho(C_{k}))$ is an ordering
of the circuit-hyperplanes in $N$.
It is now clear that
$\psi_{\trun(\mcal{C})} = \psi_{\trun(\rho(\mcal{C}))}$, so
$\mcal{R}_{k}^{r}(M)\cap \mcal{R}_{k}^{r}(N)$ contains at least one function.

For the converse, we assume $p(M)=p(N)$ and $s(M)=s(N)$,
and that $\mcal{R}_{k}^{r}(M)\cap \mcal{R}_{k}^{r}(N)$ contains a
function.
This means that we can let
$\mcal{C}_{M}=(C_{1}^{M},\ldots, C_{k}^{M})$ and
$\mcal{C}_{N}=(C_{1}^{N},\ldots, C_{k}^{N})$ be
orderings of the circuit-hyperplanes in $M$ and $N$ such that
the functions $\psi_{\trun(\mcal{C}_{M})}$ and $\psi_{\trun(\mcal{C}_{N})}$
are equal.

Assume that $M = L(G_{M},\mcal{B}_{M})$ and
$N = L(G_{N},\mcal{B}_{N})$, where $G_{M},G_{N}\in\mcal{G}$
have at least five vertices, and $\mcal{B}_{M}$ and $\mcal{B}_{N}$
contain exactly $k$ Hamiltonian cycles.
Both $G_{M}$ and $G_{N}$ contain $p:= p(M)=p(N)$ loops, and we
can assume these loops are labelled $c_{1},\ldots, c_{p}$ in both graphs.
Similarly, $G_{M}$ and $G_{N}$ have $s:= s(M)=s(N)$ thin edges, and we
assume these edges are labelled $d_{1},\ldots, d_{s}$.
Now $G_{M}$ and $G_{N}$ have $t:= r - s$ parallel pairs,
and we assume that these pairs are labelled $\{a_{1},b_{1}\},\ldots, \{a_{t},b_{t}\}$.

We will construct a permutation $\pi$ of the ground set
\[E(M)= E(N) =\{a_{i},b_{i}\}_{i=1}^{t}\cup\{c_{i}\}_{i=1}^{p}\cup\{d_{i}\}_{i=1}^{s}\]
such that:
\begin{enumerate}[label=\textup{(\roman*)}]
\item $\pi$ acts as the identity on $c_{1},\ldots, c_{p},d_{1},\ldots, d_{s}$,
\item $\pi$ takes any parallel pair $\{a_{i},b_{i}\}$ to another such pair, and
\item $\pi$ takes any circuit-hyperplane in $M$ to a circuit-hyperplane of $N$.
\end{enumerate}
Since the non-spanning circuits of $M$ and $N$ are exactly the
circuit-hyperplanes, along with sets of the
form $\{a_{i},b_{i},a_{j},b_{j}\}$, the existence of $\pi$ will show that
$M$ and $N$ are isomorphic.

Let \mcal{I} be the collection of proper subsets of $\{1,\ldots, k-1\}$.
For each $I\in\mcal{I}$ let $\pi_{I}$ be an arbitrary bijection from
$\trun(\mcal{C}_{M})(I)$ to $\trun(\mcal{C}_{N})(I)$.
This bijection exists because
$\psi_{\trun(\mcal{C}_{M})}(I) = \psi_{\trun(\mcal{C}_{N})}(I)$,
and hence
\[|\trun(\mcal{C}_{M})(I)|=|\trun(\mcal{C}_{N})(I)|.\]
Next we note that the thin edges $d_{1},\ldots, d_{s}$ are contained in every
circuit-hyperplane of $M$ and $N$.
That is, the elements $d_{1},\ldots, d_{s}$ are contained in both
$\trun(\mcal{C}_{M})(\{1,\ldots, k-1\})$ and
$\trun(\mcal{C}_{N})(\{1,\ldots, k-1\})$.
We let $\cap(M)$ stand for the intersection $\cap_{i=1}^{k}C_{i}^{M}$
and let $\cap(N)$ stand for  $\cap_{i=1}^{k}C_{i}^{N}$.
Let $\sigma$ be an arbitrary bijection from
\[\cap(M)-\{d_{1},\ldots, d_{s}\}\quad \text{to}\quad
\cap(N)-\{d_{1},\ldots, d_{s}\}.\]
Let $\id_{d}$ be the identity function on $\{d_{1},\ldots, d_{s}\}$.
Now let $\pi_{0}$ be the union
\[
\id_{d}\ \cup\ \sigma\ \cup\ \bigcup_{I\in\mcal{I}}\pi_{I}.
\]
Observe that $\pi_{0}$ is a bijection from $C_{k}^{M}$ to $C_{k}^{N}$.

We extend $\pi_{0}$ to a permutation of $E(M)=E(N)$
by insisting that it preserves parallel pairs.
To this end, we note that $C_{k}^{M}$ contains
the thin edges $d_{1},\ldots, d_{s}$, along with exactly one element
from each of the parallel pairs $\{a_{1},b_{1}\},\ldots, \{a_{t},b_{t}\}$.
We construct $\pi_{1}$, a bijection from
$\{a_{1},b_{1},\ldots, a_{t},b_{t}\}-C_{k}^{M}$ to
$\{a_{1},b_{1},\ldots, a_{t},b_{t}\}-C_{k}^{N}$.
If $x$ is in the domain of $\pi_{1}$, then $x$ is in a parallel pair
with an edge $y$ in $G_{M}$.
Moreover, $y$ is in $C_{k}^{M}$, so $\pi_{0}(y)$ is defined,
and is in $C_{k}^{N}$.
We note that $\pi_{0}(y)$ is in a parallel pair with an edge $x'$ in $G_{N}$, and
we set the image $\pi_{1}(x)$ to be $x'$.
Now we set $\pi$ to be $\pi_{0}\cup \pi_{1}\cup\id_{c}$, where
$\id_{c}$ is the identity function on $\{c_{1},\ldots, c_{p}\}$.
Thus $\pi$ is indeed a permutation of $E(M)=E(N)$, it
acts as the identity on $\{c_{1},\ldots, c_{p},d_{1},\ldots, d_{s}\}$,
and it takes any pair $\{a_{i},b_{i}\}$ to another such pair.

To complete the proof it suffices to show that $\pi$ takes any
circuit-hyperplane of $M$ to a circuit-hyperplane of $N$.
This in turn will follow if we can show that when $x$ is in
$\mcal{C}_{M}(I)$ for some $I\subseteq \{1,\ldots, k\}$, the
image $\pi(x)$ is in $\mcal{C}_{N}(I)$.
This is true when $x$ is in $\{c_{1},\ldots, c_{p}\}$, for then
$x\in\mcal{C}_{M}(I)$ implies $I=\emptyset$, and
$x=\pi(x)$ is also in $\mcal{C}_{N}(\emptyset)$.
Similarly, if $x$ is in $\{d_{1},\ldots, d_{s}\}$, then
$x\in\mcal{C}_{M}(I)$ implies $I=\{1,\ldots, k\}$, and
$x=\pi(x)$ is in $\mcal{C}_{N}(\{1,\ldots, k\})$.
So we assume that $x$ is not equal to any element
$c_{i}$ or $d_{i}$.
If $I$ contains $k$, then $x$ is in $C_{k}^{M}$, which means it
is in the domain of $\pi_{0}$.
In this case $\pi(x)=\pi_{0}(x)$ is in $\mcal{C}_{N}(I)$,
by construction of $\pi_{0}$.
Therefore we assume that $k$ is not in $I$, so $x$ is not in
$C_{k}^{M}$.
This means that $x$ is a non-loop edge that is contained
in a parallel pair $\{x,y\}$ in $G_{M}$, and furthermore
$y$ is in $C_{k}^{M}$.
Now $y$ is in exactly the circuit-hyperplanes that $x$ is not in.
In other words, $y$ is in $\mcal{C}_{M}(\{1,\ldots, k\}-I)$.
As $y$ is in the domain of $\pi_{0}$, it now follows that
$\pi_{0}(y)$ is in $\mcal{C}_{N}(\{1,\ldots, k\}-I)$.
But $\pi(x)$ is parallel to $\pi_{0}(y)$ in $G_{N}$, meaning that it
is in exactly the circuit-hyperplanes of $N$ that $\pi_{0}(y)$
is not in.
Thus it follows that $\pi(x)$ is in $\mcal{C}_{M}(I)$,
exactly as desired.
This completes the proof.
\end{proof}

\begin{lemma}
\label{scream}
Let $k$ be a positive integer.
The number of $2t$\dash element matroids in $\mcal{S}_{k}$
is at most $O(t^{2^{k-1}+2})$.
\end{lemma}

\begin{proof}
We refer to \Cref{safety}.
A Category\dash (F) matroid with $2t$ elements
is determined by giving the number of loops and the number
of coloops.
This argument shows that there are no more than $O((2t)^{2})$
such matroids, so we will henceforth disregard them.
Up to isomorphism, a $2t$\dash element matroid of the form
$M(G)$ or $M^{*}(G)$
can be determined by the number of loops and thin edges in $G$.
Therefore the number of $2t$\dash element Category\dash (D) or (E)
matroids is at most $O((2t)^{2})$, so we also disregard these classes.
There is a constant number of loopless graphs in \mcal{G} with
a bounded number of vertices.
Thus there is a constant number of graphs with \mcal{G} with
$2t$ edges and a bounded number of vertices.
So we disregard any matroids of the form
$L(G,\mcal{B})$ when $G\in\mcal{G}$ has fewer than five vertices.
Thus we have disregarded Category\dash (B) and (C) matroids,
and we now need only consider Category\dash (A) matroids
with rank at least five.
A Category\dash (A) matroid $L(G,\mcal{B})$ with at most one
circuit-hyperplane is determined up to isomorphism by the number of
loops and thin edges in $G$, so there are at most $O((2t)^{2})$
such matroids.
Therefore we may as well assume that $k$ is at least two,
and we will consider only matroids with at least two
circuit-hyperplanes.

Our arguments have shown that we need only consider
$2t$\dash element Category\dash (A) matroids
with rank at least five and at least two circuit-hyperplanes.
We categorise these matroids as having rank $r$, where
$r$ satisfies $5\leq r \leq 2t$, and having exactly
$m$ circuit-hyperplanes, where $m$ satisfies
$2\leq m \leq k$.
The number of pairs $(r,m)$ is $O(t)$, so we will be done
if we can show that the number of matroids corresponding to
the pair $(r,m)$ is at most $O(t^{2^{k-1}+1})$.
By \Cref{glance}, these matroids can be determined by a pair of
numbers from $\{0,\ldots, 2t\}$, and a function
$\psi\colon \mcal{P}(\{1,\ldots, m-1\})\to \mbb{Z}_{\geq 0}$ such that
$\sum_{I\subseteq\{1,\ldots, m-1\}}\psi(I) = r$.
The number of such functions is exactly
\[
\binom{r+2^{m-1}-1}{r} = \binom{r+2^{m-1}-1}{2^{m-1}-1}
\]
which is at most $O(r^{2^{m-1}-1})$.
This is in turn bounded by $O(t^{2^{k-1}-1})$.
There are at most $O((2t)^{2})=O(t^{2})$ ways of selecting the two numbers
in $\{0,\ldots, 2t\}$, so this leads to a bound of $O(t^{2^{k-1}+1})$ matroids
corresponding to the pair $(r,m)$, as we wanted.
\end{proof}

\begin{lemma}
\label{bottom}
Let $k\geq 2$ be an integer.
The number of $2t$\dash element excluded minors for
$\mcal{S}_{k}$ is at least $\Omega(t^{2^{k}-k-3})$.
\end{lemma}

\begin{proof}
We will assume that $t$ is at least five.

Let \mcal{I} be the collection $\{I\subseteq \{1,\ldots, k\}\colon 1\leq I \leq k-2\}$
and note that $|\mcal{I}| = 2^{k} - k - 2$.
For each $I\in \mcal{I}$ we introduce a variable $x_{I}$.
We consider non-negative integer solutions to the equation
\begin{equation}
\label{eqn2}
\sum_{I\in \mcal{I}}x_{I} = t - 2(k+1).
\end{equation}
The number of such solutions is exactly
\[
\binom{t+2^{k}-3k-5}{t-2k-2} = \binom{t+2^{k}-3k-5}{2^{k}-k-3},
\]
which is at least $\Omega(t^{2^{k}-k-3})$.

Let $\phi$ be a solution to \eqref{eqn2}, so that
$\phi$ is a function taking $\{x_{I}\}_{I\in\mcal{I}}$ to
non-negative values, and summing over the image of $\phi$
produces a total of $t-2(k+1)$.
We are going to construct a sequence
$\mcal{D} = (D_{1},\ldots, D_{k})$ of subsets of
$\{a_{1},\ldots, a_{t}\}$ in such a way that
$|\mcal{D}(I)| = \phi(x_{I})$ for each $I\in\mcal{I}$.
We do this by allocating each element in
$\{a_{1},\ldots, a_{t}\}$ to $\mcal{D}(I)$ for a unique
subset $I\subseteq \{1,\ldots, k\}$.
We start by allocating two elements to
$\mcal{D}(\emptyset)$.
These two elements are in none of the sets
$D_{1},\ldots, D_{k}$.
Next, for each $i\in\{1,\ldots, k\}$, we allocate two elements
to $\mcal{D}(\{1,\ldots, k\}-i)$.
These two elements will be in all of the sets
$D_{1},\ldots, D_{k}$ except for $D_{i}$.
Now there are $t-2(k+1)$ elements left to allocate.
We allocate no elements to
$\mcal{D}(\{1,\ldots, k\})$, so that no element of
$\{a_{1},\ldots, a_{t}\}$ is contained in all of the sets.
The remaining $n-2(k+1)$ elements in
$\{a_{1},\ldots, a_{t}\}$ are allocated to the
sets $\mcal{D}(I)$ for $I\in \mcal{I}$ according to the
function $\phi$, so that $|\mcal{D}(I)| = \phi(x_{I})$.

Next we construct subsets $\mcal{C}=(C_{1},\ldots, C_{k+1})$ of
$\{a_{1},\ldots, a_{t},b_{1},\ldots, b_{t}\}$.
We set $C_{k+1}$ to be $\{a_{1},\ldots, a_{t}\}$.
For $i\in \{1,\ldots, k\}$, we define
$C_{i}$ to be the union of $D_{i}$ and
$\{b_{j}\colon a_{j}\notin D_{i}, 1\leq j\leq k\}$.
Thus each set $C_{i}$ contains exactly one element from each pair
$\{a_{j},b_{j}\}$.
Furthermore, $\trun(\mcal{C}) = (D_{1},\ldots, D_{k})$, so
$\psi_{\trun(\mcal{C})} = \phi$.
It follows from $\mcal{D}(\{1,\ldots, k\})=\emptyset$ that
no element of $\{a_{1},\ldots, a_{t},b_{1},\ldots, b_{t}\}$
is in all of the sets $C_{1},\ldots, C_{k+1}$, and that every element is in
at least one of $C_{1},\ldots, C_{k+1}$.

Let $G$ be a graph obtained from a cycle of
length $t$ by replacing each edge with a parallel pair.
Let $\{a_{1},b_{1}\},\ldots, \{a_{t},b_{t}\}$ be the parallel
pairs in $G$.
Let \mcal{B} be the class of Hamiltonian cycles in $G$ with
edge-sets $C_{1},\ldots, C_{k+1}$.
We claim that \mcal{B} is a linear class.
Let us assume otherwise.
Any theta-subgraph in $G$ consists of a Hamiltonian
cycle with one additional edge.
So if \mcal{B} is not a linear class,
then there are two sets $C_{i}$ and $C_{j}$ such that
$C_{i}-C_{j}$ contains a single element.
But two elements of $\{a_{1},\ldots, a_{t}\}$ are in none
of the sets $D_{1},\ldots, D_{k}$, so these two elements
are in $C_{k+1}$ but none of $C_{1},\ldots, C_{k}$.
So $i$ is not $k+1$, and hence $i$ is in $\{1,\ldots, k\}$.
There are two elements of $\{a_{1},\ldots, a_{t}\}$
that are in all of the sets $D_{1},\ldots, D_{k}$ other than $D_{i}$.
This means that two elements in $\{b_{1},\ldots, b_{t}\}$
are in none of the sets $C_{1},\ldots, C_{k+1}$
except for $C_{i}$, and this is a contradiction.
Therefore \mcal{B} is a linear class,
as we claimed.

We let $M$ be the spike $L(G,\mcal{B})$.
Since $M$ has $k+1$ circuit-hyperplanes,
and $t\geq 5$, it follows without difficulty from \Cref{galaxy} that
$M$ is not in $\mcal{S}_{k}$.
However, since every element of $M$ is in at least one
circuit-hyperplane, and avoids at least one
circuit-hyperplane, deleting or contracting any element
from $M$ produces a minor $L(G',\mcal{B}')$,
where $G'$ is in \mcal{G}, and $\mcal{B}'$ contains
at most $k$ Hamiltonian cycles.
So $M$ is indeed an excluded minor for $\mcal{S}_{k}$.

For each solution to \eqref{eqn2} we construct an excluded
minor for $\mcal{S}_{k}$, as detailed above.
Some of these excluded minors may be isomorphic.
But all excluded minors constructed in this way have
no triangles and no triads, so the functions $p$ and $s$ return zero.
Now \Cref{glance} implies that the excluded minors
are isomorphic if and only if they have the same images under
$\mcal{R}_{k+1}^{t}$.
So any isomorphism class amongst the constructed excluded minors
is no larger than the image of a matroid under
$\mcal{R}_{k+1}^{t}$, which is at most $(k+1)!$.
Since there are $\Omega(t^{2^{k}-k-3})$ solutions to
\eqref{eqn2}, and $k$ is constant with respect to $t$, it follows that
the number of excluded minors is at least $\Omega(t^{2^{k}-k-3})$,
as claimed.
\end{proof}

\begin{proof}[Proof of \textup{\Cref{memory}}.]
\Cref{cheese} implies that $\Gamma_{\mcal{S}_{k}}(2t+1)=0$
for all sufficiently large values of $t$.
So $\Gamma_{\mcal{S}_{k}}(n)$ does not tend to one, and hence
$\mcal{S}_{k}$ is certainly not strongly fractal.
However, by \Cref{scream,bottom} we see that for sufficiently
large values of $t$ we have
\[
\Gamma_{\mcal{S}_{k}}(2t)\geq
\frac{c_{1}t^{2^{k}-k-3}}{c_{2}t^{2^{k-1}+2}+c_{1}t^{2^{k}-k-3}}
=\frac{1}{(c_{2}/c_{1})t^{-2^{k-1}+k+5}+1}
\]
for some constants $c_{1}$ and $c_{2}$.
Since $k\geq 5$, we see that $-2^{k-1}+k+5$ is negative.
Thus $\Gamma_{\mcal{S}_{k}}(2t)$ tends to one as $t$ tends to
infinity, meaning that $\mcal{S}_{k}$ is weakly fractal.
\end{proof}

\end{document}